\documentclass[10pt,a4paper]{amsart}
\usepackage{amsmath,amsthm,amssymb,mathrsfs,hyperref}
\usepackage{aliascnt}
\usepackage{charter}
\usepackage{fullpage}

\usepackage{xcolor}
\definecolor{doom}{rgb}{0,0,0.68}
\hypersetup{
	unicode=true,
	colorlinks=true,
	citecolor=doom,
	linkcolor=doom,
	anchorcolor=doom
}

\makeatletter
\expandafter\g@addto@macro\csname th@plain\endcsname{%
		\thm@notefont{\bfseries}
	}%
\makeatother

\newtheorem{theorem}{Theorem}
\newaliascnt{example}{theorem}
\newtheorem{example}[example]{Example}
\aliascntresetthe{example}

\newtheorem*{remark}{Remark}
\newtheorem*{question}{Question}

\newcommand{\axiom}[1]{\mathsf{#1}} 
\newcommand{\ZFC}{\axiom{ZFC}}

\newcommand{\DC}{\axiom{DC}}
\newcommand{\ZF}{\axiom{ZF}}
\newcommand{\Ord}{\axiom{Ord}}
\newcommand{\HOD}{\axiom{HOD}}
\DeclareMathOperator{\cf}{cf}
\DeclareMathOperator{\dom}{dom}
\newcommand{\proves}{\mathrel{\vdash}}

\address{\textbf{Einstein Institute of Mathematics}\\
Edmond J. Safra Campus, Givat Ram\\
The Hebrew University of Jerusalem.\\
Jerusalem, 91904, Israel}
\author{Asaf Karagila}
\email{karagila@math.huji.ac.il}
\urladdr{http://boolesrings.org/asafk}
\date{\today}

\title{Absolutely Choiceless Proofs}

\begin{document}
\begin{abstract}
We study a well-known technique of using absoluteness for giving choice-free proofs to some statements which are known to be provable with the axiom of choice. The idea is to reduce the problem to an inner model where the axiom of choice holds and use absoluteness. We examine the complexity of the sentences that this technique can be applied to, and show that many of the theorems in basic partition calculus have the adequate complexity for this technique to apply.
\end{abstract}
\maketitle
\section{Introduction}
Recall the L\'{e}vy hierarchy of formulas in the language of set theory. We say that a quantifier $Q$ is bounded if it is of the form $(Qx\in y)$. We define $\Sigma_0$ formula to be a formula equivalent to one in which all the quantifiers are bounded, $\Pi_0$ is a negation of $\Sigma_0$ and $\Delta_0$ are statements which are both $\Pi_0$ and $\Sigma_0$. Note that all three classes defined here are the same.

If $\Sigma_n,\Pi_n,\Delta_n$ were defined, then we define $\Sigma_{n+1}$ to be formulas equivalent to $\exists x\varphi(x)$, where $\varphi$ is a $\Pi_n$ formula; $\Pi_{n+1}$ is a formula equivalent to a negation of $\Sigma_{n+1}$; and $\Delta_{n+1}$ is a formula which is both $\Sigma_{n+1}$ and $\Pi_{n+1}$.

To read more about these one should consult, for example, \cite[Chapter~13]{Jech:ST2003}. And we will assume that the reader is familiar with the basic statements about this hierarchy.

One of the classical theorems is that if $M\subseteq V$ are two models of $\ZF$ with the same ordinals (and $M$ is transitive in $V$), then every $\Sigma_1$ statement true in $M$ is true in $V$ (in which case we say that the statement is \textbf{upwards absolute}), and every $\Pi_1$ statement true in $V$ is true in $M$ (in which case we say that the statement is \textbf{downwards absolute}). It follows that $\Delta_1$ statements are absolute between two models of $\ZF$ with the same ordinals. Moreover, if we assume that $V\models\ZF+\DC$ then every $\Pi_1$ sentence is absolute between $V$ and $L$, and therefore $\Pi_1$ sentences are absolute between all models of $\ZF+\DC$.

But absoluteness can also be used to prove statements without the axiom of choice in the following manner.

\begin{example}[Erd\H{o}s-Dushnik-Miller Theorem] Assuming that in $\ZFC$ it is provable that $\kappa\to(\kappa,\omega)^2$ for every uncountable $\kappa$, then it is provable in $\ZF$.
\end{example}
\begin{proof}
Suppose that $V\models\ZF$, and $\kappa$ is uncountable in $V$. Let $c$ be a coloring of $[\kappa]^2$ in $V$, and consider $L[c]$ which is a model of $\ZFC$. Since $c\subseteq L$ we have $c\in L[c]$.

By the fact that $\kappa$ is uncountable in $V$ it is certainly uncountable in $L[c]$ and so there is a subset $X\subseteq\kappa$ witnessing the truth of the Erd\H{o}s-Dushnik-Miller Theorem, and so $X\in V$ and it is a witness for the wanted homogeneity. 
\end{proof}

Of course that is an indirect result, rather than directly constructing the sets we use the fact that this is true in a definable inner model. And in \cite{MathOverflow40507} Andr\'es Caicedo asks whether or not a direct proof can be given of a consequence of the above theorem (which can be proved in a similar fashion). The answer is yes, and an argument due to Clinton Conley and is given on the same page.

Before we proceed we give another striking example with a seemingly very different nature.

\begin{example}[Magidor] If $\cf(\omega_1)=\cf(\omega_2)=\omega$ then $0^\#$ exists.
\end{example}
\begin{proof}
Let $\kappa=\omega_1^V$. Let $A\subseteq\kappa$ be a cofinal sequence of length $\omega$. Then in $L[A]$, $\kappa$ is singular. Let $\alpha$ be $(\kappa^+)^{L[A]}$, then in $V$ we have that $\cf(\alpha)=\omega$ as well. Let $B\subseteq\alpha$ be a cofinal sequence of order type $\omega$.

Now in $L[A,B]$ we have that both $\kappa$ and $\alpha$ are singular, and therefore the successor of $\kappa$ in $L[A,B]$ is some $\beta<\alpha$. If so, $L[A,B]$ knows that $L[A]$ miscalculates the successor of a singular cardinal, $\kappa$, and therefore it knows that $0^\#$ exists, and so $V$ knows that as well.
\end{proof}

\section{The General Theorem}

\begin{question}[Raghavan]
To what sort of level in the L\'{e}vy hierarchy can we use the above trick, and show that if a statement is provable from $\ZFC$ then it is provable from $\ZF$?\footnote{December 11th, 2013. Math Department lounge in Jerusalem. The question came up when we were having coffee and talking about the Erd\H{o}s-Dushnik-Miller Theorem in a choiceless context.}
\end{question}

It is trivial that every $\Sigma_1$ statement satisfies that criteria. If $\varphi$ is a $\Sigma_1$ statement, and $V$ satisfies $\ZF$, then $L\subseteq V$ satisfying $\ZFC$, and therefore $L\models\varphi$. By absoluteness $V\models\varphi$, and so we get that $\ZF$ proves $\varphi$ as well.

But one can notice that in order to use the trick of reducing back to $L[A]$ we needed $A$ to be a set of ordinals, or constructibly coded into a set of ordinals. We introduce the following definition of an \textbf{ordinal bounded quantifier}, $(\exists^\Ord x)\varphi(x)$ is the statement $\exists x(x\subseteq\Ord^{<\omega}\land\varphi(x))$. That is to say that $x$ is a set of tuples of ordinals, and satisfies $\varphi$. We similarly define $\forall^\Ord$. 

The statement $x\subseteq\Ord^{<\omega}$ is a $\Delta_0$ statement (with $x$ as a parameter saying that all the elements of $x$ are functions whose domain is a finite ordinal, and whose range is a set of ordinals), therefore changing existential and universal quantifiers to ordinal bounded quantifiers does not increase (nor decreases) the complexity of the formula in the L\'{e}vy hierarchy.

\begin{remark}
We can require only that $x$ is a subset of $L$ in order to immediately have that $x\in L[x]$, or even that $x$ is a subset of a model of $\ZFC$, such as $\HOD$. However writing that $x\subseteq L$ will increase the complexity of the formula, whereas just requiring that $x$ is a set of tuples of ordinals does not.
\end{remark}

\begin{theorem}
If $\varphi$ is of the form $(\forall^\Ord x)\psi(x)$, where $\psi(x)$ is upwards absolute for models of $\ZF$ with the same ordinals, then $\ZF\proves\varphi$ if and only if $\ZFC\proves\varphi$.
\end{theorem}
\begin{proof}
Clearly every statement provable from $\ZF$ is provable from $\ZFC$. 

Suppose that $\varphi$ is as above and $\ZFC\proves\varphi$. Let $V$ be a model of $\ZF$, and let $a\subseteq\Ord^{<\omega}$ in $V$. Consider $L[a]$, which is a model of $\ZFC$. We have $a\subseteq L$, so $a\in L[a]$. By the assumption, $L[a]\models(\forall^\Ord x)\psi(x)$, in particular for $L[a]\models\psi(a)$, and since $\psi$ is absolute we have $V\models\psi(a)$. Therefore $V\models\varphi$, so $\ZF\proves\varphi$ as wanted.
\end{proof}

Some examples for statements satisfying the requirements of the theorem are $\Pi_1$ statements whose universal quantifiers are ordinal bounded, and in fact $\Pi_2$ statements whose outer universal quantifiers are ordinal bounded. This givens a partial improvement over the absoluteness of $\Pi_1$ sentences, since no assumption of $\DC$ is needed now. The following example reveals the limitation of the theorem when it comes to existential quantifiers.

\begin{example}
The statement ``The real numbers can be well-ordered'' is a $\Sigma_2$ sentence with all its quantifiers ordinal bounded quantifiers, which is provable from $\ZFC$ but not from $\ZF$.
\end{example}
\begin{proof}
Cohen proved the consistency of $\ZF$ with the failure of this statement (see \cite{Cohen:1964}, as well \cite[Theorem~14.36]{Jech:ST2003}). And it is true in $\ZFC$, since every set can be well-ordered. We calculate the complexity of the statement to see that it is indeed $\Sigma_2$.

The real numbers can be well-ordered if and only if there is an ordinal $\alpha$ and a function $f\colon\alpha\to 2$ such that for every $a\colon\omega\to2$ there is some $\beta<\alpha$ such that for all $n$, $f(\beta+n)=a(n)$.

\begin{align*}
&(\exists^\Ord\alpha)(\exists^\Ord f)(\forall^\Ord x)\\
&\Bigg(\alpha\text{ is an ordinal}\land (f\colon\alpha\to 2)\land\Big((x\colon\omega\to 2)\rightarrow(\exists\beta\in\alpha)(\forall n\in\omega)x(n)=f(\beta+n)\Big)\Bigg).
\end{align*}

The statements that $\alpha$ is an ordinal and $f$ is a function are both $\Delta_0$, generally the quantification on $\omega$ and the ordinal arithmetic are $\Delta_1$, therefore the entire expression inside the quantifiers is $\Delta_1$ with parameters $\alpha,f$ and $x$. So the entire statment is $\Sigma_2$ as wanted.
\end{proof}

\section{Analysis of the Examples}
While it is clear how the first example falls into the absoluteness theorem above, it is unclear how Magidor's theorem works out here. After all, $\ZFC$ always proves that both $\omega_1$ and $\omega_2$ are regular. But a careful analysis of the proof shows that in fact we can write the following statement: For every two sets of ordinals $A$ and $B$, if the following holds,
\begin{enumerate}
\item In $L[A]$ all the elements of $A$ are cardinals ($\Pi_1$ statement), and $\sup A$ is singular ($\Sigma_1$ statement).
\item In $L[A]$ $\sup B=(\sup A)^+$ ($\Delta_2$ statement).
\item In $L[A,B]$ there is an injection from $\sup B$ into $\sup A$ ($\Sigma_1$ statement).
\end{enumerate}
Then $0^\#$ exists ($\Sigma_2$ statement). 

Since being an element of $L[A]$ is $\Delta_1$ with $A$ as a parameter, the above is a $\Delta_2$ statement with parameters $A$ and $B$, and the existence of $0^\#$ is a $\Sigma_2$ statement, moreover every quantifier which increases the complexity beyond $\Delta_1$ is ordinal bounded (in fact the only place were we use non-ordinal bounded quantifiers is in limiting our quantified objects to $L[A]$ or $L[A,B]$ which is $\Delta_1$). So as a whole we have an implication between two $\Sigma_2$ statements, ``There exists $A$ and $B$ such that some $\Delta_2$ property in $A$ and $B$'' then ``$0^\#$ exists''. The whole implication can be written as a $\Pi_3$ statement, whose universal quantifiers are ordinal bounded, and the internal $\Sigma_2$ statement is upwards absolute. Therefore $\ZF$ proves the implication as well, and Magidor's proof is given by showing that from assuming that $\omega_1$ and $\omega_2$ are singular, there are witnesses that the implication is not vacuous.

We finish with the proof that essentially all the basic results of infinitary combinatorics about partition calculus which are provable in $\ZFC$ are provable in $\ZF$.

\begin{theorem}
Every statement of the form ``For every coloring of a subset of $[\kappa]^{<\omega}$ in $\lambda$ colors, there is a subset of order type/cardinality $\alpha$ which is homogeneous/anti-homogeneous'' is $\Pi_2$ with all its quantifier being ordinal bounded quantifiers (and with $\kappa,\lambda$ as parameters, and possibly $n$ in $[\kappa]^n$).
\end{theorem}
\begin{proof}
The coloring is a function from strictly increasing finite functions to $\kappa$ into $\lambda$, so quantifying ``every coloring'' is a $\forall^\Ord f$ and the fact that $f$ is a coloring is $\Delta_0$ (with parameters $\kappa,\lambda$ and $n$ if the functions have a fixed or bounded length).

Next we say that there exists a set of ordinals $A\subseteq\kappa$, and there exists a bijection or an order preserving bijection from $A$ to $\alpha$, or the relevant ordinal appearing in the theorem, that would be two consecutive $\exists^\Ord$ quantifiers, and a $\Delta_0$ formula.

Finally we say that $A$ is homogeneous which is to say that every finite function in $\dom f$ whose range includes only elements from $A$ is mapped to a single element. Or that $A$ is anti-homogeneous if such $f$ is injective. Both of these properties are $\Delta_0$. 

And so we get that the statement is of the form \[(\forall^\Ord f)(\exists^\Ord A)(\exists^\Ord g)\varphi(f,A,g,\kappa,\alpha,\lambda,n),\] where $\varphi$ is a $\Delta_0$ statement in which we write out the above facts.
\end{proof}

\section{Acknowledgments}
The author would like to thank D.\ Raghavan for inspiring the idea of a general framework of these absoluteness results, and to U.\ Abraham for his help in finding the elegant $\Sigma_2$ statement in Example~4, and to Haim Horowitz and Yair Hayut for their help in correcting a mistake in the author's interpretation of that $\Sigma_2$ statement in a previous version of this manuscript.

\bibliographystyle{amsalpha}
\providecommand{\bysame}{\leavevmode\hbox to3em{\hrulefill}\thinspace}
\providecommand{\MR}{\relax\ifhmode\unskip\space\fi MR }
\providecommand{\MRhref}[2]{%
  \href{http://www.ams.org/mathscinet-getitem?mr=#1}{#2}
}
\providecommand{\href}[2]{#2}

\end{document}